\documentclass[11pt,leqno]{amsart}
\usepackage{amssymb,amsmath,amsthm,amsfonts}
\usepackage{graphicx}

\newcommand{\field}[1]{\mathbb{#1}}
\newcommand{\cW}{{\mathcal W}}

\newcommand{\bA}{{\mathbf A}}
\newcommand{\bv}{{\mathbf v}}
 \DeclareMathOperator{\diam}{diam}

\newcommand{\N}{\field{N}}                      
\newcommand{\R}{\field{R}}                      
\newcommand{\C}{\field{C}}                      

\newcommand{\dimAtht}{\dim_A^\theta}

\newcommand{\eps}{\epsilon}

\newcommand{\loc}{{\scriptstyle{loc}}}

\def\Barint_#1{\mathchoice
          {\mathop{\vrule width 6pt height 3 pt depth -2.5pt
                  \kern -8pt \intop}\nolimits_{#1}}%
          {\mathop{\vrule width 5pt height 3 pt depth -2.6pt
                  \kern -6pt \intop}\nolimits_{#1}}%
          {\mathop{\vrule width 5pt height 3 pt depth -2.6pt
                  \kern -6pt \intop}\nolimits_{#1}}%
          {\mathop{\vrule width 5pt height 3 pt depth -2.6pt
                  \kern -6pt \intop}\nolimits_{#1}}}

\theoremstyle{plain}
\newtheorem{theorem}{Theorem}
\newtheorem{corollary}[theorem]{Corollary}

\newtheorem{proposition}[theorem]{Proposition}

\theoremstyle{definition}

\numberwithin{theorem}{section} \numberwithin{equation}{section}

\makeatletter
\@namedef{subjclassname@2020}{\textup{2020} Mathematics Subject Classification}
\makeatother
\title[QR distortion of dimensions]{Quasiregular distortion of dimensions}
\date{\today}
\author{Efstathios-K. Chrontsios-Garitsis}
\subjclass[2020]{Primary 37F31; Secondary 28A80, 30C62}
\keywords{Quasiregular mappings, holomorphic mappings, Assouad dimension, Assouad spectrum}
\address{Department of Mathematics \\ University of Tennessee, Knoxville \\ 1403 Circle Dr \\ Knoxville, TN 37966}
\email{echronts@utk.edu, echronts@gmail.com}

\begin{document}

\maketitle

\begin{abstract}
We investigate the distortion of the Assouad dimension and (regularized) spectrum of sets under planar quasiregular maps. The respective results for the Hausdorff and upper box-counting dimension follow immediately from their quasiconformal counterparts by employing elementary properties of these dimension notions (e.g.~countable stability and Lipschitz stability). However, the Assouad dimension and spectrum do not share such properties. We obtain upper bounds on the Assouad dimension and spectrum of images of compact sets under holomorphic and planar quasiregular maps by studying their behavior around their critical points. As an application, the invariance of porosity of compact subsets of the plane under quasiregular maps is established.
\end{abstract}

\section{Introduction}

Quasiregular maps are considered a natural generalization of holomorphic maps, especially in higher dimensions. However, their significance is not restricted to the higher dimensional setting. They often play a central role in complex dynamics and are an essential part of a set of techniques called ``quasiconformal surgery", which allows us to ``glue" otherwise rigid maps together while preserving some of their properties. Quasiconformal surgery was first introduced by Douady and Hubbard for polynomial-like maps (\cite{Douady}, \cite{DouadyHub}) and later generalized further by Shishikura in \cite{Shishikura} for rational maps in order to study the number of stable regions and Herman rings associated to their dynamics (see also \cite{NuriaBook} for a thorough exposition of these techniques).

The plethora of fractals encountered in dynamics and other areas has been a strong motive to determine quantities that help in classifying these complicated sets. It is in fact a core element of Fractal Geometry to study different dimension notions for fractals (see \cite{falconer2014}). One notion that has  been recently used in a range of different areas  is the Assouad dimension, which was first introduced by Assouad in \cite{Assouad83} as a tool to investigate the embedability of metric spaces into Euclidean spaces. Moreover, Fraser and Yu in \cite{fy:assouad-spectrum} introduced the notion of Assouad spectrum, a collection of dimension values interpolating between the upper box-counting dimension and the Assouad dimension. This notion has been recently studied extensively, since it provides more information on sets which exhibit a ``gap" between their box and Assouad dimensions. The book by Fraser \cite{Fraser2020} is one of the most complete references to date for properties and applications of the Assouad dimension and spectrum.

The study of how dimension notions change under quasiconformal mappings has been an interesting problem at the intersection of Analysis and Fractal Geometry. Gehring and V\"ais\"al\"a \cite{GV} first established how quasiconformal maps distort the Hausdorff dimension and Kaufman \cite{Kaufman} showed the same bounds hold for the quasiconformal distortion of the upper box-counting dimension. Tyson and the author \cite{Chronts} proved that the Assouad dimension and regularized Assouad spectrum of sets change similarly under quasiconformal maps, a result that played a pivotal role in quasiconformally classifying polynomial spirals.

One way to define quasiregular maps is with the same dilatation bound condition as in quasiconformal maps, but without the assumption of the map being a homeomorphism (see Section 2 for definitions). It is a natural question to ask whether the ``bounded stretching" properties of quasiregular and quasiconformal maps are enough for these dimension distortion bounds, or the fact that quasiconformal maps are homeomorphisms plays an important role as well. In the cases of the Hausdorff and upper box-counting dimensions the analytic properties of quasiconformal maps, which the class of quasiregular maps also shares, are indeed enough to establish the same upper bounds on the dimensions of images of sets. This can be seen even more clearly in the case of a planar quasiregular map, which can be written as a composition of a holomorphic and a quasiconformal map (see \cite{Rickman}, \cite{LehtoVir}). Indeed, the Hausdorff dimension does not change under holomorphic maps, since it is countably stable and the critical points are at most countably many. Hence, a quasiregular map only changes the Hausdorff dimension of a set by however much the respective quasiconformal map changes it, which is established by the result of Gehring and V\"ais\"al\"a \cite{GV}. Similarly, the upper box-counting dimension is Lipschitz stable, i.e., it does not increase under Lipschitz maps, implying that the upper box-counting dimension of a compact set would only increase by at most the amount the respective quasiconformal map increases it, which is bounded by the result of Kaufman \cite{Kaufman}.

However, things are different in the case of the Assouad dimension and spectrum, which are neither countably stable nor Lipschitz stable (see \cite{Fraser2020} p.~18, 48). The requirement of these notions to check all points and all appropriate pairs of scales in the set creates the need for additional information in order to use counting arguments in the source of a map to bound the dimension of the image (see Section 2 for definitions). Indeed, such additional properties were essential in the proof of the quasiconformal distortion bounds of the Assouad dimension and spectrum in \cite{Chronts} (see Corollaries 2.2 and 2.3 for rigorous statements). However, analyzing how planar quasiregular maps behave around their critical points provides information of similar nature. Leveraging this understanding, we were able to show the following two main results for the quasiregular distortion of the Assouad dimension and spectrum, respectively.

\begin{theorem}\label{thm:Adim}
	Given $K\geq1$, let  $f:\Omega \rightarrow \C$ be a non-constant $K$-quasiregular mapping defined on a domain $\Omega \subset \C$ and let $E\subset \Omega$ be a compact set with $\dim_A E=\alpha\in (0,2)$. Then we have
	$$ \dim_A f(E)\leq \frac{2K\alpha}{2+(K-1)\alpha}<2.$$
\end{theorem}

\begin{theorem}\label{th:Aspec}
	Let $f:\Omega\to\C$ be a $K$-quasiregular map as in Theorem \ref{thm:Adim}. For $t>0$ and a compact set $E\subset\Omega$,  we have
	\begin{equation}\label{eq:main2}
		\dim_A^{\theta(t)} f(E) \le \frac{2K\alpha_{t/K}}{2+(K-1)\alpha_{t/K}},
	\end{equation}
	where $\theta(t):=\frac{1}{1+t}$ and $\alpha_{t}:=\dim_{A,reg}^{\theta(t)}(E)$.
\end{theorem}

It is important to note that Theorems \ref{thm:Adim} and \ref{th:Aspec} were not known even in the case of $K=1$, i.e., in the case of non-constant holomorphic maps. In fact, a substantial portion of the proofs is dedicated to establishing these bounds for holomorphic maps. This might seem unexpected, especially for Theorem \ref{thm:Adim}, given the extensive interest in Assouad dimension and the significance of holomorphic maps throughout many areas in mathematics. However, the non-trivial arguments we had to employ in order to show that holomorphic maps do not increase the Assouad dimension perhaps illustrate why such bounds were not previously established. 

Moreover, the bound in Theorem \ref{thm:Adim} is sharp, which is an outcome of the sharpness of the corresponding bound for quasiconformal maps. Namely, Gehring and V\"ais\"al\"a in \cite{GV} constructed for each $\alpha\in (0,2)$, a set $E_\alpha$ of Hausdorff dimension equal to $\alpha$, and a $K$-quasiconformal map $f:\C\rightarrow\C$,  such that the Hausdorff dimension of $f(E_\alpha)$ is equal to $\beta=\frac{2K\alpha}{2+(K-1)\alpha}$. For the specific sets they constructed, the Assouad dimensions of $E_\alpha$ and $f(E_\alpha)$ are equal to $\alpha$ and $\beta$, respectively. As a result, one could then say that Theorem \ref{th:Aspec} is also sharp in some sense, based on the same examples. However, when it comes to the Assouad spectrum, what one considers ``sharp" can be subjective. In particular, it would perhaps be more interesting to demonstrate sharpness by constructing a set $E$ with non-constant Assouad spectrum, and a $K$-quasiregular (or quasiconformal) map $f$, such that \eqref{eq:main2} holds with equality. Such a construction is not yet known.

An immediate consequence of Theorem \ref{thm:Adim} is the following.

\begin{corollary}\label{cor:porous}
	Porosity of compact subsets of $\C$ is preserved under planar quasiregular mappings.
\end{corollary}

This can be seen as the quasiregular version of V\"ais\"al\"a's result on quasisymmetric invariance of porosity in \cite{vai:porosity}. The proof of Corollary \ref{cor:porous} is immediate by Theorem \ref{thm:Adim} and the characterization of porosity proved by Luukkainen in \cite{Luukkainen}, i.e., a subset of $\R^n$ is porous if, and only if, its Assouad dimension is strictly less than $n$.

This paper is organized as follows. Section 2 reviews the definitions of the Assouad dimension and spectrum and the respective results on their distortion under quasiconformal maps. In Section 3 we prove Theorems \ref{thm:Adim} and \ref{th:Aspec} by investigating the distortion of dimensions under holomorphic maps. Section 4 contains examples and further remarks on our main results.

\smallskip

\paragraph{\bf Acknowledgments.} 

Part of this research was conducted during the author's visit to University of Jyv\"askyl\"a. The author wishes to thank Sylvester Eriksson-Bique for the invitation and the Jenny and Arttu Wihuri Foundation for partial support through the ``Quasiworld Network" during this visit. The author also wishes to thank Kai Rajala and Pekka Pankka for the very interesting discussions, as well as the anonymous referee whose remarks greatly improved the manuscript.

\section{Background}

A map $f:\Omega \to \C$ defined in a domain $\Omega\subset \C$ is said to be {\it $K-$quasiregular} for $K\geq1$, if $f$ lies in the local Sobolev space $W^{1,2}_\loc(\Omega)$ and the inequality
\begin{equation}\label{eq:qc-defn}
	|Df|^2 \le K \det Df
\end{equation}
holds a.e.\ in $\Omega$. Here $Df$ denotes the (a.e.\ defined) differential matrix and $|\bA|=\max\{|\bA(\bv)|:|\bv|=1\}$ denotes the operator norm of a matrix $\bA$. 
Moreover, if $f$ is a homeomorphism onto its image, then it is  a {\it $K-$quasiconformal} map.

The dimension notions we are focusing on are the Assouad dimension and the (regularized) Assouad spectrum. We will first establish the notation we are following. For a point $z\in \C$ and a positive scale $R>0$ we denote by $D(z,R)$ the open disc centered at $z$ of radius $R$. For an arbitrary set $E \subset \C$ and a smaller scale $r\in (0,R)$ we denote by $N(D(z,R) \cap E,r)$ the smallest number of sets of diameter at most $r$ needed to cover $D(z,R) \cap E$. 

We can now define the {\it Assouad dimension} of a set $E\subset \C$ as follows
$$
\dim_A E = \inf \left\{\alpha>0 \,:\, {\exists\,C>0\mbox{ s.t. } N(D(z,R) \cap E,r) \le C (R/r)^{\alpha} \atop \mbox{ for all $0<r\le R$ and all $z \in E$}} \right\}.
$$ Furthermore, given $\theta\in (0,1)$ we can define the \textit{regularized Assouad $\theta$-spectrum} of $E$ as follows
\begin{equation}\label{def:Assouad-spectrum}
	\dim_{A,reg}^\theta E = \inf \left\{\alpha>0 \,:\, {\exists\,C>0\mbox{ s.t. } N(D(z,R) \cap E,r) \le C (R/r)^{\alpha} \atop \mbox{ for all $0<r\le R^{1/\theta}< R< 1$ and all $z \in E$}} \right\}.
\end{equation} It should be noted that this is a slight modification of the original spectrum defined by Fraser and Yu in \cite{fy:assouad-spectrum}, where they used the relation $r=R^{1/\theta}$ for the pair of scales instead of $r\leq R^{1/\theta}$. The regularized Assouad spectrum can also be found as the \textit{upper} Assouad spectrum in the literature and is closely related to the one defined in \cite{fy:assouad-spectrum} (see \cite{Fraser2020} p.~32). 
In particular, if $\dim^\theta E$ denotes the original spectrum of a set $E$ using the relation $r=R^{1/\theta}$, as defined in \cite{fy:assouad-spectrum}, then 
$$
\dim_{A,reg}^{\theta_0} E= \sup\{ \dim^\theta E: 0<\theta<\theta_0 \},
$$ for all $\theta_0\in (0,1)$.
However, since we are only using the notion where the scales are related by an inequality, we will be calling the set of all dimension values in \eqref{def:Assouad-spectrum} for all $\theta\in (0,1)$ the \textit{Assouad spectrum} of $E$, and  use the simplified notation $\dimAtht E$.

It is often easier to work with specific sets, and especially dyadic squares, instead of arbitrary sets of diameter at most $r$ to cover a set. 
In particular, for a set $E \subset \C$ and $z \in E$, $R>0$ we consider the axes-parallel square $Q:=Q(z,R) \subset \C$, we subdivide $Q$ into $2^2$ essentially disjoint sub-squares, each with side length equal to half of the side length of $Q$, and then we subdivide each of those squares in the same fashion, and so on. Let $\cW(Q)$ denote the collection of all such squares obtained at any level of the construction, and let $\cW_m(Q)$ denote the collection of all squares obtained after $m$ steps. We will denote by $N_d(D(z,R) \cap E,m)$ the number of dyadic squares in $\cW_m(Q)$ needed to cover $D(z,R) \cap E$. The following property of the Assouad spectrum proved in \cite[Proposition 2.5 (ii)]{Chronts} allows us to interchange between dyadic squares and arbitrary sets in the definition of $\dimAtht E$.

\begin{proposition}\label{prop:Assouad-spectrum-technical-proposition}
	Let $E$ be a bounded subset of $\C$ and let $\theta\in(0,1)$. The Assouad spectrum value $\dim_A^\theta(E)$ is equal to the infimum of all $\alpha>0$ for which there exists $C>0$ so that
		$$
		N_d(D(z,R) \cap E,m) \le C 2^{m\alpha}
		$$
		for all $z \in E$ and all $m\in \N$ such that $0<2^{-m}R\le R^{1/\theta} < R < 1$.
\end{proposition}

We now recall the results on quasiconformal distortion of the Assouad dimension and spectrum, which will be used in Section 3.
\begin{theorem}[\cite{Chronts}, Theorem 1.2]\label{thQC:Adim}
	Given $K\geq1$, let  $f:\Omega \rightarrow \Omega'$ be a $K$-quasiconformal mapping between domains $\Omega, \Omega' \subset \C$ and $E\subset \Omega$ be a compact set with $\dim_A E=\alpha\in (0,2)$. Then we have
	$$ \dim_A f(E)\leq \frac{2K\alpha}{2+(K-1)\alpha}<2.$$
\end{theorem}

\begin{theorem}[\cite{Chronts}, Theorem 1.3]\label{thQC:Aspec}
	Let $f:\Omega\to\Omega'$ be a $K$-quasiconformal map as in Theorem \ref{thQC:Adim}. For $t>0$ and a compact set $E\subset\Omega$,  we have
	\begin{equation}\label{eq:QCmain2}
		\dim_A^{\theta(t)} f(E) \le \frac{2K\alpha_{t/K}}{2+(K-1)\alpha_{t/K}},
	\end{equation}
	where $\theta(t):=\frac{1}{1+t}$ and $\alpha_{t}:=\dim_{A}^{\theta(t)}(E)$.
\end{theorem}
It should be noted that the two theorems above were proved for quasiconformal maps defined between domains in $\R^n$, $n\geq 2$. In that setting, however, the upper bounds only implicitly depend on the dilatation $K$ for $n>2$, through the higher integrability exponent of $K$-quasiconformal maps. The upper bounds are explicit in the two-dimensional case only thanks to Astala's result \cite{Astala} calculating the higher integrability exponent for planar quasiconformal maps. We refer the reader to the discussion in the introduction of \cite{Chronts} for more details.

\section{Distortion of Assouad dimension and spectrum}
Suppose $f:\Omega \rightarrow \C$ is a non-constant quasiregular map and let $E$ be a subset of the domain $\Omega\subset \C$ with $\alpha= \dim_A E$. To provide an upper bound on $\dim_A f(E)$ we need to provide a suitable upper bound for $N(D(w,R') \cap f(E),r')$, where $w\in f(E)$, $0<r'<R'$ are arbitrary.  As mentioned in the introduction, however, it can be challenging to find the appropriate disc $D(z,R)\subset \Omega$ in the source in order to use our knowledge of coverings of $D(z,R)\cap E$ and count the number of sets that we need to map with $f$ to cover $D(w,R')\cap f(E)$, especially close to points where $f$ fails to be a local homeomorphism.
A fundamental property of planar quasiregular maps, which is called \textit{Sto\"ilow factorization} of quasiregular maps, is used to simplify this task. More specifically, if $f$ is $K$-quasiregular for some $K\geq 1$, then there is a holomorphic map $h$ and a $K$-quasiconformal map $g$ such that $f=h\circ g$ (see for instance \cite{LehtoVir} p.~247). We will use this property to reduce the problem to the distortion of the Assouad dimension and spectrum under holomorphic maps. Indeed, in view of Theorem \ref{thQC:Aspec} it is enough for the quasiregular distortion of the Assouad spectrum to prove the following:

\begin{theorem}\label{thm:Holom}
	Let $h:\Omega \rightarrow \C$ be a non-constant holomorphic map in a domain $\Omega \subset \C$ and $E \subset \Omega$ be a compact set. Then $\dim_A^\theta h(E)\leq \dim_A^\theta E$ for all $\theta\in (0,1)$.
\end{theorem}

\begin{proof}
	Since $E$ is compact, it contains at most finitely many critical points of $h$. If $E$ contains no critical points of $h$, then there are $m_1, m_2>0$ such that $|h'(z)|\in [m_1, m_2]$ for all $z\in E$, which implies $\dimAtht h(E)= \dimAtht E$ for all $\theta\in (0,1)$ by bi-Lipschitz stability of the Assouad spectrum (see \cite{Fraser2020} p.~18, 49).
	
	Suppose $E$ contains critical points of $h$, which we denote by $c_1, \dots, c_m$. For any $c_j$, there are $\eps_j<1/10$ and conformal maps $\phi_j$, $\psi_j$ such that
	\begin{equation}\label{eq:conformal}
	(\phi_j\circ h\circ \psi_j)(z)=z^{d_j},
	\end{equation} for all $z\in D(c_j, \eps_j)$, where $d_j\geq2$ is the degree of $f$ at $c_j$ (see for instance \cite{Palka} p.~346).
	Fix $\theta\in (0,1)$. By finite stability of the Assouad spectrum we have
	\begin{equation}\label{eq:finite_stab}
		\dimAtht h(E)= \max\{\dimAtht h(\tilde{E}), \dimAtht h(E\cap D(c_1, \eps_1/100)), \dots,  \dimAtht h(E\cap D(c_m, \eps_m/100))\},
	\end{equation} where $\tilde{E}=E\setminus \cup_{j=1}^m D(c_j,\eps_j/100)$. However, there are $\tilde{m}_1, \tilde{m}_2>0$ such that $|h'(z)|\in [\tilde{m}_1, \tilde{m}_2]$ for all $z\in \tilde{E}$, which implies that $\dimAtht h(\tilde{E})=\dimAtht \tilde{E}\leq \dimAtht E$. Hence, it is enough to show that $\dimAtht h(E\cap D(c_j, \eps_j/100))\leq \dimAtht (E\cap D(c_j, \eps_j/100))$ for all $j= 1, \dots, m$. In fact, by bi-Lipschitz invariance of the Assouad spectrum and \eqref{eq:conformal}, \eqref{eq:finite_stab}, it is enough to prove that the map $h(z)=z^d: D(0,\eps)\rightarrow \C$ with $d\geq 2$ and $\eps\in(0,1/10)$  satisfies $\dimAtht h(E) \leq \dimAtht E$ for compact $E\subset D(0,\eps/100)$ with $0\in E$.

	Fix arbitrary $\alpha > \dimAtht E$ and $p>2$. We will show that $\dimAtht h(E)\leq \beta := \frac{p \alpha}{p-2+\alpha} $ and take $\alpha \searrow \dimAtht E$, $p\rightarrow \infty$ to finish the proof. 
	
	Let $w\in h(E)$ and $R'>0$. In fact, due to $E\subset D(0,\eps/100)$ it is enough to consider $R'<2\eps^d /100^d$. If $0\notin \overline{D(w,R')}\cap E$ then $\dimAtht (h(E)\cap D(w,R'))=\dimAtht (E\cap h^{-1}(D(w,R')))\leq \dimAtht E <\alpha<\beta$ and it follows that there is $C_0>0$ such that 
	$$
	N(D(w,R')\cap h(E), r')\leq C_0 \left( \frac{R'}{r'} \right)^\beta
	$$ for all $r'\in (0, (R')^{1/\theta})$. A relation of this kind in the case $0\in \overline{D(w,R')}\cap E$ would finish the proof. Suppose $0\in \overline{D(w,R')}\cap E$ and set
	$$
	r_j':= 2^{\frac{-j\alpha}{\beta}}R'
	$$
	for all $j\geq j_0$, where $j_0\in \N$ is the smallest positive integer such that $2^{\frac{-j_0\alpha}{\beta}}R'\leq (R')^{1/\theta}<2^{\frac{(-j_0+1)\alpha}{\beta}}R'$, and 
	\begin{align*}
		R &:= (2R')^{1/d} \\
		r_j&:= 2^{-j}R.
	\end{align*} It is enough to show that there is a constant $C>0$ such that for all $j\geq j_0$ we have
	\begin{equation}\label{eq:N_target_desired}
	N(D(w,R')\cap h(E), r_j')\leq C \left( \frac{R'}{r_j'} \right)^\beta = C 2^{j \alpha},
	\end{equation} which implies that $\dimAtht h(E)\leq \beta$.
	
	Fix $j\geq j_0$. Note that the images of sets in a covering of $ D(0,R)\cap E$ naturally form a covering of $h(D(0,R)\cap E)$, which is also a covering of $D(w,R')\cap h(E)$, since $D(w,R')\cap h(E)\subset h(D(0,R)\cap E)$ by choice of $R$ and the fact that $0\in \overline{D(w,R')}$. Hence,  we will cover $D(0,R)\cap E$ with dyadic sub-squares $\{ Q_k^j \}$ of $Q(0,R)=[-R,R]\times [-R,R]\subset D(0,\eps)$ of side length $r_j$ and take their images under $h$ to cover $D(w,R')\cap h(E)$. Moreover, by $j\geq j_0$ we have the inequality $r_j'\leq (R')^{1/ \theta}$, which implies
	\begin{align*}
		2^{\frac{-j\alpha}{\beta}}R'&\leq (R')^{1/\theta}\\
		-\frac{j\alpha}{\beta}&\leq \left(\frac{1}{\theta}-1\right)\log_2(R') \\
		j&\geq \left(\frac{1}{\theta}-1\right) \frac{\beta}{\alpha}(-\log_2(R')) \\
		&\geq\left(\frac{1}{\theta}-1\right) \frac{1}{d}(-\log_2(R')-1).
	\end{align*} The last inequality is important, since it ensures that for $j\geq j_0$ we have 
	$$
	r_j=2^{-j} R= 2^{-j} (2R')^{1/d}\leq ((2R')^{1/d})^{1/\theta}=R^{1/\theta}.
	$$ This allows us to  use $\alpha>\dimAtht E$ and Proposition \ref{prop:Assouad-spectrum-technical-proposition} to count how many sub-squares $\{ Q_k^j\}$ of $Q(0,R)$ we actually need. In particular, there is $C>0$ such that
	\begin{equation}\label{eq:dyad_source}
	N_d(D(0,R)\cap E, j)\leq C 2^{j\alpha}.
	\end{equation} 
	If all squares $Q_k^j$ have images under $h$ of diameter at most $r_j'$, i.e., $\diam h(Q_k^j)\leq r_j'$ for all $k$, then by \eqref{eq:dyad_source} the covering $\{h(Q_k^j)\}$ of $D(w,R')\cap h(E)$ consists of at most $C 2^{j\alpha}=C \left( \frac{R'}{r_j'} \right)^\beta$, as needed for the upper bound on $\dimAtht h(E)$.
	
	Suppose this is not the case and there are squares $Q_k^j$ for which $\diam h(Q_k^j)>r_k'$. The plan is to sub-divide these squares and their ``children" until eventually their images under $h$ have diameter at most $r_j'$. We know this is possible by uniform continuity of $h$, but we need to count how many times we have to sub-divide, i.e., how many images of such squares we will end up including in our covering. For $\ell\geq j$ we call a sub-square $Q_{k,n}^{\ell}$ of $Q_k^j$ of side length $r_{\ell}$ $j$-\textbf{major} if $\diam h(Q_{k,n}^{\ell})>r_j'$ and $j$-\textbf{minor} otherwise.
	We need to count all $j$-major sub-squares of all levels $\ell$, subdivide them all once and count the number of resulting $j$-minor squares, which is an upper bound for $N(D(w,R')\cap h(E), r_j')$.
	
	For all $j$-major $Q_{k,n}^{\ell}$, we have that $h\in W^{1,p}(Q_{k,n}^{\ell})$. Applying Morrey-Sobolev inequality on $h$ on such a major sub-square gives
	$$
	\diam h(Q_{k,n}^{\ell}) \leq C_S (\diam (Q_{k,n}^{\ell}))^{1-2/p} \left( \int_{Q_{k,n}^{\ell}} |Dh|^p \right)^{1/p}.
	$$ 
    Note that the inequality holds for complex functions in the same way as  for planar mappings between domains in $\R^2$ (see for instance \cite{LehtoVir} for more details on this correspondence). 
    For what follows, we will not keep track of multiplying constants that do not depend on the scales $R$, $R'$ or the levels $j$, $\ell$ and keep denoting all of them by $C_S$. Since $Q_{k,n}^{\ell}$ is $j$-major, the above inequality implies
	$$
	r_j'\leq C_S\, r_\ell^{1-2/p} \left( \int_{Q_{k,n}^{\ell}} |Dh|^p \right)^{1/p},
	$$ which by definition of $r_j'$ and $r_\ell$ leads to
	$$
	2^{-\frac{j\alpha p}{\beta}} (R')^p\leq C_S\, 2^{-\ell(p-2)}R^{p-2}  \int_{Q_{k,n}^{\ell}} |Dh|^p .
	$$ Since all $j$-major $\ell$-level sub-squares are essentially disjoint, we can sum over all such squares in the above inequality and have
	$$
	2^{-\frac{j\alpha p}{\beta}} (R')^p M(\ell)\leq C_S\, 2^{-\ell(p-2)}R^{p-2}  \int_{\bigcup_{k,n}Q_{k,n}^{\ell}} |Dh|^p ,
	$$ where $M(\ell)$ is the number of all $\ell$-level $j$-major squares. Note that $k$ runs through all $j$-major sub-squares $Q_k^j$ of $Q(0,R)$ and $n$ runs through all $j$-major sub-squares of $Q_k^j$ of side length $r_\ell$. Summing over all levels $\ell\geq j$ and noting that $R=(2R')^{1/d}$ results in
	$$
	2^{-\frac{j\alpha p}{\beta}} (R')^p \sum_{\ell\geq j}M(\ell)\leq C_S\, \sum_{\ell\geq j} 2^{-\ell(p-2)}(R')^{\frac{p-2}{d}} \left( \int_{\bigcup_{k,n,\ell}Q_{k,n}^{\ell}} |Dh|^p \right).
	$$ However, the above inequality along with $\sum_{\ell\geq j} 2^{-\ell(p-2)}\leq C_1 2^{-j(p-2)}$ and
	$$
	 \int_{\bigcup_{k,n,\ell}Q_{k,n}^{\ell}} |Dh|^p \leq  \int_{D(0, \sqrt{2} (2R')^{1/d})} |Dh|^p  \leq C_2 (R')^{2/d} \left((R')^{1/d}\right)^{p(d-1)}
	$$ implies that
	\begin{equation}\label{eq:almost_done}
	2^{-\frac{j\alpha p}{\beta}} (R')^p \sum_{\ell\geq j}M(\ell)\leq C_S\, 2^{-j(p-2)} (R')^{p/d-2/d} (R')^{2/d} \left((R')^{1/d}\right)^{p(d-1)}.
	\end{equation}
	Note that by considering the part of $E$ in $D(0,\eps/100)$, i.e., the part of $E$ ``deep" inside $D(0,\eps)$, allows for all arbitrary meaningful $R'$ to be such that $D(0, \sqrt{2} (2R')^{1/d})\subset D(0,\eps)$. This guarantees that in the general holomorphic map case, the map would still behave like $z^d$ in all discs considered in the arguments above. 
	
	It is now clear that all terms with $R'$ cancel out in \eqref{eq:almost_done} and because
	$$
	-j(p-2)+j\alpha p/\beta=-j(p-2)+j(p-2+\alpha)=j \alpha
	$$ due to $\beta=p\alpha/(p-2+\alpha)$, we have
	$$
	\sum_{\ell\geq j}M(\ell)\leq C_S \, 2^{j\alpha}.
	$$ Since $\sum_{\ell\geq j}M(\ell)$ is the number of all $j$-major sub-squares of $Q(0,r)$, subdividing all of them once leads to a covering consisting of only $j$-minor squares, which are at most
	$$
	4\sum_{\ell\geq j}M(\ell)\leq 4C_S \, 2^{j\alpha}.
	$$ As a result, their images form a covering of $D(w,R')\cap h(E)$ by sets of diameter at most $r_j'$, which implies
	$$
	N(D(w,R')\cap h(E), r_j')\leq 4C_S \, 2^{j\alpha}= C \left( \frac{R'}{r_j'} \right)^\beta
	$$ as needed to complete the proof.

\end{proof}

\textit{Proof of Theorem \ref{th:Aspec}:} By the decomposition $f=h\circ g$, where $h$ is holomorphic, $g$ is $K$-quasiconformal and Theorems \ref{thm:Holom} and \ref{thQC:Aspec} the proof is immediate. \\ \qed\\

Note that Theorem \ref{thm:Adim} does not immediately follow by Theorem \ref{th:Aspec} itself, since by taking $t\rightarrow 0^+$, i.e., $\theta \rightarrow 1^-$, in \eqref{eq:main2} would only imply the desired upper bound for the so-called \textit{quasi-Assouad dimension} $\dim_{qA} h(E):=\lim_{\theta\rightarrow 1^-}\dimAtht h(E)$, which can differ from the Assouad dimension (see \cite{Fraser2020}, Section 5.3 for instance). However, the proof of Theorem \ref{thm:Adim} also relies on the fact that holomorphic maps do not increase the Assouad dimension, which can be proved almost identically to how Theorem \ref{thm:Holom} was proved.

\begin{theorem}\label{thm:HolomA}
	Let $h:\Omega \rightarrow \C$ be a non-constant holomorphic map in a domain $\Omega \subset \C$ and $E \subset \Omega$ be a compact set. Then $\dim_A h(E)\leq \dim_A E$.
\end{theorem}

\begin{proof}
	Following the reductions in the proof of Theorem \ref{thm:Holom}, it is enough to prove that the map $h(z)=z^d: D(0,\eps)\rightarrow \C$ with $d\geq 2$ and $\eps\in(0,1/10)$  satisfies $\dim_A h(E) \leq \dim_A E$ for compact $E\subset D(0,\eps/100)$ with $0\in E$.
	
	Let $\alpha>\dim_A E$, $p>2$ and set $\beta=\frac{p\alpha}{p-2+\alpha}$. Let $w\in h(E)$ and $R'\in(0,2 \eps^d/100^d)$. The proof follows exactly the same way as the proof of Theorem \ref{thm:Holom}, with the only difference being that the scales
	\begin{align*}
		r_j'&:= 2^{\frac{-j\alpha}{\beta}}R'\\
		R &:= (2R')^{1/d} \\
		r_j&:= 2^{-j}R.
	\end{align*} are now defined for all $j\in \N$. This does not change any of the arguments in the proof of Theorem \ref{thm:Holom} and we can similarly prove that there is $C>0$ such that
	$$
	N(D(w,R')\cap h(E),r_j')\leq C 2^{j\alpha}= C \left( \frac{R'}{r_j'} \right)^\beta.
	$$ Since $y$ and $R'$ were arbitrary, this shows that $\dim_A h(E)\leq \beta$ and taking $\alpha\searrow \dim_A E$, $p\rightarrow \infty$ completes the proof. 
\end{proof}

\textit{Proof of Theorem \ref{thm:Adim}:} It follows directly by the decomposition $f=h\circ g$ and Theorems \ref{thm:HolomA} and \ref{thQC:Adim}. \qed

\section{Final Remarks}

It is important to note that in Theorem \ref{thm:Adim},  the restriction to compact sets lying inside the domain is necessary. Set $h(z)=-\text{Log} z$ defined on $\C \setminus \{ t\in \R: t\leq0 \}$, where $\text{Log}$ denotes the principal logarithm. This is a holomorphic map that takes the sequence $$E=\{ e^{-n}: n\in \N \}\subset \C\setminus \{ t\in \R: t\leq0 \}$$ onto the positive integers $h(E)=\N$. However, $\dim_A E=0<1=\dim_A h(E)$. 

Moreover, the map $g(z)=1/z: \C\setminus\{0\}\rightarrow\C\setminus\{0\}$ is holomorphic, and $E=\N$ has constant Assouad spectrum equal to $0$. On the other hand, by \cite[p.~42]{Fraser2020}, the image $g(E)=\{ 1/n: n\in \N\}$ satisfies
$$
\dimAtht g(E)= \min\left\{ \frac{1}{2(1-\theta)}, 1 \right\},
$$ for all $\theta\in (0,1)$, clearly not satisfying the bound \eqref{eq:main2}. Hence, the compactness of the set $E$ lying inside the domain is a necessary condition for Theorem \ref{th:Aspec} as well.

Another remark is regarding the strictness of the quasiregular (and holomorphic) distortion of the Assouad spectrum. Consider the compact set $E=\{ 1/n: n\in \N \}\cup \{0\}$ and the holomorphic map $h(z)=z^2$. We have $h(E)=\{ 1/n^2: n\in \N \}\cup \{0\}$, and 
$$
\dimAtht E= \min\left\{ \frac{1}{2(1-\theta)}, 1 \right\},
$$
$$
\dimAtht h(E)= \min\left\{ \frac{1}{3(1-\theta)}, 1 \right\},
$$ for all $\theta\in (0,1)$, due to \cite[p.~42]{Fraser2020}. As a result, we have the strict inequality 
$$\dim_A^{1/2} h(E)=2/3<1=\dim_A^{1/2} E.$$

We finish the paper by recalling that the quasiconformal distortion Theorems \ref{thQC:Adim} and \ref{thQC:Aspec} \mbox{actually} hold for higher dimensions. However, the techniques used to prove Theorems \ref{thm:Adim} and \ref{th:Aspec} heavily relied on properties of holomorphic maps and the fundamental fact that planar quasiregular maps can be written as a composition of holomorphic and quasiconformal mappings. In higher dimensions, while the ``analytic" part of the proof could be reproduced similarly to the arguments used for Theorems \ref{thQC:Adim} and \ref{thQC:Aspec} in \cite{Chronts}, the ``topological" part of going back from the image to the source in an appropriate way is greatly hindered by the unpredictable properties of the branch set of quasiregular maps, which can be of higher Hausdorff dimension (see Theorems 1.1 and 1.3 in \cite{BonkQR}). While there are results on the behavior of higher-dimensional quasiregular maps around the branch set (see for instance \cite{BlochEremenko}, \cite{BlochKai1}, \cite{BlochKai2}, \cite{EggYolkQR}), and even partial results regarding the porosity invariance around the branch set (see \cite{KaiOnninenBranchImagePorous} and \cite{SarvasDimHofBranch}), it can still be difficult to establish a strong enough connection between the coverings we need in the image and the ones we can count in the source. This was not a problem in two dimensions because of the discreteness of the critical set and the well-understood local behavior of holomorphic maps around critical points.


\end{document}